\documentclass[12pt]{article}
\usepackage{graphicx}
\usepackage[leqno]{amsmath}
\usepackage{mathptmx}
\usepackage{latexsym,amsmath,amssymb,amsfonts,amsthm}
\numberwithin{equation}{section}

\newtheorem{theorem}{Theorem}[section]
\newtheorem{lemma}[theorem]{Lemma}

\newtheorem{proposition}[theorem]{Proposition}
\newtheorem{remark}[theorem]{Remark}

\newenvironment{proof of theorem 6.2}{{\it Proof of Theorem 6.2}.}{{\hfill $\square$%
    \hskip - \parfillskip}}

\begin{document}

\setcounter{page}{1}
\title{ Isoperimetric inequalities for eigenvalues \\by inverse mean curvature flow }
\author{Fangcheng Guo$^1$, Guanghan Li$^2$, and Chuanxi Wu$^1$}
\date{}
\protect\footnotetext{\!\!\!\!\!\!\!\!\!\!\!\!\! {\bf MSC 2000:}
53C40, 53C44, 35k55.
\\
{\bf ~~Key Words:} Laplacian operator, isoperimetric inequality, inverse mean curvature flow.\\
The research is partially supported by NSFC (no. 11171096),
RFDP (no. 20104208110002), and Funds for Disciplines Leaders
of Wuhan (no. Z201051730002).}
\maketitle ~~~\\[-15mm]
\begin{center}{\footnotesize 1. School of Mathematics and Statistics, Hubei University, Wuhan, 430062, P. R. China
\\2. School of Mathematics and Statistics, Wuhan University, Wuhan, 430072, P. R. China
\\E-mail: gfch2006@163.com, ghli@whu.edu.cn, cxwu@hubu.edu.cn}
\end{center}
\begin{abstract} \noindent
By studying the monotonicity of the first nonzero eigenvalues of Laplace and $p$-Laplace operators on a closed convex hypersurface $M^n$ which evolves under inverse mean curvature flow in $\mathbb{R}^{n+1}$,  the isoperimetric lower bounds for both eigenvalues were founded.

\end{abstract}

\markright{\sl\hfill \hfill}

\section{Introduction}

The isoperimetric type problem, which is always expressed as whether an inequality about some geometric quantities achieves optimization in the geodesic sphere case among a class of bounded domains in Riemannian manifolds, plays an important role in geometry \cite{Wang-Xia}. A basic theme of it is to generalize the classic isoperimetric inequalities in a Euclidean space to a higher dimensional Riemannian manifold, especially to a space form (cf. \cite{Chavel}).

Another significant theme of it is isoperimetric inequalities for eigenvalues of Laplace and $p$-Laplace. Rayleigh, as the first one studying this problem, posed the famous conjecture \cite{Rayleigh}: \emph{Let $\Omega\subset \mathbb{R}^n$ be a bounded domain in $\mathbb{R}^n$, $B$ be some ball in $\mathbb{R}^n$ which has the same volume with $\Omega$, $\lambda_1(\Omega)$ and $\lambda_1(B)$ be the first nonzero Dirichlet eigenvalues of $\Omega$ and $B$, respectively, then
$$
\lambda_1(\Omega)\geq\lambda_1(B),
$$
with equality holding iff $\Omega=B$.
}

This conjecture has been settled by Faber and Krahn \cite{Faber,Krahn}, well known as Rayleigh-Faber-Krahn inequality. This problem for the $p$-Laplace operator is solved by Bhattacharya \cite{Bhattacharya}. The corresponding problem to the first nonzero Neumann eigenvalue, called free membrane problem, is solved by Szeg\"{o} and Weinberger \cite{Weinberger}, but with an opposite sign of inequality, known as Szeg\"{o}-Weinberger inequality. A natural problem is whether there exist similar inequalities for the closed eigenvalue problem, namely, the eigenvalue for compact hypersurface without boundary embedded in $\mathbb{R}^{n+1}$. This is one of motivations of this paper. Recently, Wang and Xia \cite{Wang Q} proved that for a closed hypersurface $M^{n}$ embedded in $\mathbb{R}^{n+1}$ which encloses a bounded domain $\Omega$, then
\begin{eqnarray*}
\lambda_1(M)\leq\frac{n\textmd{vol}(M)}{(n+1)\textmd{vol}(\Omega)}\left(\frac{\omega_{n+1}}{\textmd{vol}(\Omega)}\right)^{\frac{1}{n+1}},
\end{eqnarray*}
with equality holding iff $M$ is an $n$-sphere, where $\textmd{vol}(M)$ and $\textmd{vol}(\Omega)$ denote respectively the area of the closed hypersurface $M$ and the volume of the domain $\Omega$, and $\omega_{n+1}$ is the volume of a unit $(n+1)$-sphere.

In 2012, for the same case, Santhanam \cite{SANTHANAM} proved that
\begin{eqnarray*}
\frac{\lambda_1(M^n)}{\lambda_1(\mathbb{S}^n(R))}\leq\left(\frac{\textmd{vol}(M^n)}{\textmd{vol}(\mathbb{S}^n(R))}\right)^2,
\end{eqnarray*}
where $R>0$ is such that $\textmd{vol}(\Omega) = \textmd{vol}(B(R))$, and also equality holding iff $M$ is an $n$-sphere. Subsequently, Binoy and Santhanam generalized this result in \cite{Binoy} to the hypersurface in Riemannian manifold with bounded sectional curvature. It is noting that all of above results are obtained by the traditional method of centers of gravity, which seems an effective way to treat the upper bound of the first eigenvalue.

In this paper, we shall study this problem to obtain a lower bound for the first nonzero eigenvalues of $p$-Laplace and Laplace of closed hypersurfaces using a new geometric inequality which is preserved under inverse mean curvature flow.

It is known that various mean curvature flows, such as inverse mean curvature flow \cite{Gerhardt Rn,Urbas}, forced mean curvature flow \cite{Li-Sa}, are effective tools to study isoperimetric problems, as they, or after rescaled, always deform hypersurfaces to a sphere maintaining some feature. Recently, the quermassintegral isoperimetric inequalities have been studied intensively by inverse mean curvature flow, see \cite{Ge-Wang1,Ge-Wang2,guan ineq,Wang-Xia}. Meanwhile, many authors investigated the property of the first nonzero eigenvalues of the Laplace and $p$-Laplace operators under various mean curvature flows. Firstly,
 Zhao~\cite{zhao09} derived the evolution equation of the first nonzero Laplace eigenvalue under Huisken's compact contracting mean curvature flow, and obtained some interesting monotonic quantities. Subsequently, he \cite{zhao12} proved the differentiability of the first eigenvalue of $p$-Laplace under powers of $m$-th mean curvature flow, and also obtained some interesting monotonic quantities related to it. Recently, Mao~\cite{Mao} proved the differentiability for the first eigenvalue of $p$-Laplace under forced mean curvature flow, and obtained the monotonicity of the first nonzero eigenvalues for both Laplace and $p$-Laplace operators along the forced mean curvature flow under some assumptions.

In this paper, by considering the monotonicity of the first nonzero eigenvalues of Laplace and $p$-Laplace operators on the closed hypersurface evolving under inverse mean curvature flow, we obtain following isoperimetric inequalities.

\begin{theorem}\label{th1.1}
Let $M^n$ be a closed hypersurface embedded in $\mathbb{R}^{n+1}$, satisfying $H> 0$ and $h_{ij}\geq\frac{\alpha}{2} H g_{ij}$ for a constant $\alpha \in [0, \frac{2}{n}]$. Assume $\lambda_{1,p}(M^n)$ and  $\lambda_{1}(M^n)$ are the first nonzero closed eigenvalues of $p$-Laplace and Laplace operators respectively on the hypersurface $M^n$, then for some sphere $\mathbb{S}^n(R)$ embedded in $\mathbb{R}^{n+1}$ such that $\textmd{vol}(\mathbb{S}^n(R))=\textmd{vol}(M^n)$, we have
$$
\lambda_{1,p}(M^n)\geq C^{-1}(n,p,\alpha)\lambda_{1,p}(\mathbb{S}^n(R)), ~~~\text{where}~~~ C(n,p,\alpha)=\exp\left[\frac{p}{\alpha}\left(\frac{1}{n}-\frac{\alpha}{2}\right)\right].
$$
When $p=2$
$$
\lambda_{1}(M^n)\geq C^{-1}(n,\alpha)\lambda_{1}(\mathbb{S}^n(R)), ~~~\text{where}~~~ C(n,\alpha)=\exp\left[\frac{2}{\alpha}\left(\frac{1}{n}-\frac{\alpha}{2}\right)\right].
$$
Furthermore, equality holds in the two inequalities above iff M is an $n$-sphere of radius $R$.
\end{theorem}

The rest of this paper is organized as follows: In section 2, we collect some fundamental facts of eigenvalues and
 then derive the evolutions of the first nonzero eigenvalues of $p$-Laplace and Laplace under the general curvature evolution of hypersurfaces in Euclidean spaces. In section 3, We especially consider the case of inverse mean curvature flow and get a new geometric inequality preserving under the flow, which plays a key role in the proof of the main result. In the last section, we prove the monotonicity of the first nonzero eigenvalues and complete the proof of the main theorem by using the inequality obtained in section 3.

\section{Evolutions of $\lambda_1(t)$ and $\lambda_{1,p}(u,t)$ }

Let $(M^n,g)$ be an $n$-dimensional compact Riemannian manifold, $\nabla$ and $\textmd{div}$ denote the covariant derivative and divergence operator on $M^n$, respectively. Under local coordinates  $\{x_{1},\cdots,x_{n}\}$, the $p$-Laplace operator for some smoothly function $u=u(x)$ on $M^n$ is defined as
\begin{eqnarray*}
\Delta_{p}u=\textmd{div}(|\nabla u|^{p-2}\nabla u)=\frac{1}{\sqrt{\det(g_{ij})}}\sum\limits_{i,j=1}^{n}\frac{\partial}{\partial{x_{i}}}\left(\sqrt{\det(g_{ij})}g^{ij}|\nabla{u}|^{p-2}
\frac{\partial{u}}{\partial{x_{j}}}\right),
\end{eqnarray*}
where $|\nabla{u}|^{2}=g^{ij}\nabla_iu\nabla_iu.$

If there exists some non-trivial smooth function $u$ defined on $M^n$ satisfying
\begin{eqnarray*}
\triangle_pu+\lambda(M^n)|u|^{p-2}u=0
\end{eqnarray*}
for some positive constant $\lambda(M^n)$, we call $\lambda(M^n)$ the eigenvalue of the the $p$-Laplace operator on $M^n$. It is known that the $p$-Laplace operator has discrete eigenvalues, although whether it only has discrete eigenvalue is still unknown when $p\neq2$ \cite{Mao}. Its fist nonzero eigenvalue is denoted by $\lambda_{1,p}(M^n)$ here.

The first nonzero eigenvalue of the the $p$-Laplace operator on $M^n$ can also be defined as
\begin{equation}\label{df1}
\lambda_{1,p}(M^n)=\min_{0\neq u\in
W^{1,p}_{0}(M^n)}\frac{\int_{M^n}|\nabla
u|^{p}d\mu}{\int_{M^n}|u|^{p}d\mu},
\end{equation}
where $W^{1,p}_{0}(M)$ denotes the Sobolev space given by the closure of $C^{\infty}$ functions with compact support on $M^n$ for the norm
$$\|u\|^{p}_{1,p}=\int_{M^n}|u|^{p}d\mu+\int_{M^n}|\nabla u|^{p}d\mu.$$

When $p=2$, above all become the standard definitions about general Laplace operator, the first nonzero eigenvalue for it is denoted by $\lambda_1(M^n)$ below.

For an $n$-dimensional smooth and compact manifold $\Sigma^n$, let $X_0: \Sigma^n\rightarrow \mathbb{N}^{n+1}(K)$ be a smooth immersion, where $\mathbb{N}^{n+1}(K)$ is a space form of constant sectional curvature $K$. We consider a family of smooth immersions $X: \Sigma^n\times [0,T)\rightarrow \mathbb{N}^{n+1}(K)$ satisfying the following curvature evolution
\begin{equation}\label{gf}
\frac{\partial X}{\partial t}=f(X(\cdot,t))\nu, \qquad X(\cdot, 0)=X_0(\cdot), \qquad t\in[0,T),
\end{equation}
where $f$ is a curvature function of points $X(\cdot, t)$ on the evolving hypersurface $M^n_t=X(\Sigma^n, t)$ with $M^n=X(\Sigma^n, 0)$, $\nu$ is the outward unit normal vector field , and $I$ is the maximal existence interval.

In the following, $g_{ij}$ denotes the induced metric on $M^n_t$, $H$ denotes  mean curvature, $h_{ij}$, $|A|$ are respectively the second fundamental form and its norm, and $d\mu$ denotes the volume element on $M^n_t$.

We denote by $\lambda_{1,p}(t)$ and $\lambda_1(t)$ the first nonzero eigenvalues of $p$-Laplace and Laplace operators respectively on the evolving hypersurface $M^n_t$. In order to discuss the monotonicity of $\lambda_{1,p}(t)$ and $\lambda_1(t)$ under the inverse mean curvature flow \eqref{flow1}, we need calculate the evolution equations. But unfortunately, we don't known whether $\lambda_{1,p}(t)$ is differentiable under this flow or not when $p\neq2$, so a similar method to one in \cite{cao,cao1,Mao,zhao12} will be used to avoid this  difficulty. Precisely, we define a smooth function
\begin{equation}\label{def2.7}
\begin{split}
\lambda_{1,p}(u,t)=-\int_{M^n_t}\Delta_{p}u(x,t)\cdot{u(x,t)}d\mu=\int_{M^n_t}|\nabla{u}|^{p}d\mu
\end{split}
\end{equation}
on the evolving hypersurface $M^n_t$ whenever the flow \eqref{gf} exists, where $u=u(x,t)$ is an arbitrary smooth function satisfying
$$\int_{M^n_t}|u|^pd\mu=1,~\qquad \int_{M^n_t}|u|^{p-2}ud\mu=0.$$

If $u=u(x,t)$ is the eigenfunction of $\lambda_{1,p}(t)$, then we have
\begin{eqnarray*}
\lambda_{1,p}(u,t)=-\int_{M^n_t}\Delta_{p}u(x,t)\cdot{u(x,t)d\mu}=\lambda_{1,p}(t)\int_{M^n_t}|u(x,t)|^{p}d\mu=\lambda_{1,p}(t).
\end{eqnarray*}
From the definition of $\lambda_{1,p}(t)$ in \eqref{df1} and $\lambda_{1,p}(u,t)$ in \eqref{def2.7}, we have $\lambda_{1,p}(u,t)\geq\lambda_{1,p}(t)$, which is an important fact below.

Zhao in \cite{zhao09} derived the evolution equation for $\lambda_1(t)$ under the convex contracting mean curvature flow in Euclidean space $\mathbb{R}^{n+1}$, and also for $\lambda_{1,p}(u,t)$ under the powers of $m$-th mean curvature flow in $\mathbb{R}^{n+1}$ in detail. Mao in \cite{Mao} gave the evolution equations for both $\lambda_1(t)$ and $\lambda_{1,p}(u,t)$ under the forced mean curvature flow in $\mathbb{R}^{n+1}$. From their process, we find that only the evolution for the induced metric on the evolving hypersurface and the Codazzi property for the second fundamental form is used. As a result, we will derive both of them on the  hypersurface $M^n_t$ embedded in space form $\mathbb{N}^{n+1}(K)$ evolving under a general flow $\frac{\partial X}{\partial t}=f(X(\cdot,t))\nu$ by a similar method.

\begin{proposition}\label{th1.3}
Let $\lambda_{1,p}(t)$ and  $\lambda_{1}(t)$ be respectively the first nonzero closed eigenvalues of $p$-Laplace and Laplace operators on the $n$-dimensional compact hypersurface $M_t$ which evolved by flow~\eqref{gf} in a space form $\mathbb{N}^{n+1}(K)$. Assume that $u=u(x, t)$ is the corresponding eigenfunction of $\lambda_{1,p}(t)$ at time $t\in I$ satisfying $\int_{M_{t}}|u|^{p}d\mu=1$, and $\lambda_{1,p}(u,t)$ is a smooth function defined in~\eqref{def2.7}. Then at time $t$, we have

\begin{eqnarray}\label{evp}
\frac{d}{dt}\lambda_{1,p}(u,t)=-p\int_{M_t}f|\nabla
u|^{p-2}u_{i}u_{j}h^{ij}d\mu+\int_{M_{t}}fH |\nabla
u|^{p}d\mu-\lambda_{1,p}(t)\int_{M_{t}}fH|u|^{p}d\mu.
\end{eqnarray}
When $p=2$, let $u=u(x, t)$ be the eigenfunction of $\lambda_{1}(t)$ on $M_t$, we have
\begin{eqnarray}\label{ev2}
\frac{d}{dt}\lambda_{1}(t)=-2\int_{M_t}fu_{i}u_{j}h^{ij}d\mu+\int_{M_{t}}fH
|\nabla u|^{2}d\mu-\lambda_{1}(t)\int_{M_{t}}fH|u|^{2}d\mu.
\end{eqnarray}
\end{proposition}

To prove the above evolutions, first we need the following lemma in \cite{guan flow}.

\begin{lemma}~\cite{guan flow}
Let $M^n_t$ be a smooth family of closed hypersurfaces in a space form
$\mathbb{N}^{n+1}(K)$ evolving under the flow \eqref{gf}, then\vskip.1in

 $(i)$ $\frac{\partial g_{ij}}{\partial t}=2fh_{ij}$;\vskip.1in

 $(ii)$  $ \frac{\partial h_{ij}}{\partial t}=-\nabla_i\nabla_j f + f (h^2)_{ij}-Kfg_{ij}$;\vskip.1in

 $(iii)$ $\frac{\partial H}{\partial t} = -\Delta f - f |h|^2-nKf$;\vskip.1in

 $(iv)$ $\frac{\partial d\mu_g}{\partial t}=fHd\mu_g$;\vskip.1in
\noindent
where $(h^2)_{ij}=h_{ik}h_{jl}g^{jk}$.
\end{lemma}

\noindent
\emph{\textbf{Proof of Proposition 2.1}}

Let $\{x_1, x_2,\cdots, x_n\}$ be the local normal coordinates on $M_t$, and denote by $B=|\nabla u|^{p-2}, B_t=\frac{\partial B}{\partial t}, \nabla_iu=u_i, \nabla_i\nabla_ju=u_{ij}$ for convenience. Taking derivatives with respect to time $t$ on both sides of~\eqref{def2.7}, we have
\begin{equation}\label{2.9}
\begin{split}
\frac{\partial}{\partial t}\lambda_{1,p}(u,t)&=\frac{\partial}{\partial t}\int_{M_t}\triangle_{p}u(x,t)u(x,t)d\mu\\
&=\frac{\partial}{\partial t}\int_{M_t}\mathrm{div}(B\nabla u)ud\mu\\
&=\frac{\partial}{\partial t}\int_{M_t}(g^{ij}B_{i}u_{j}+B\triangle
u)ud\mu\\&=\int_{M_t}\left(\frac{\partial}{\partial
t}g^{ij}B_{i}u_{j}+g^{ij}B_{ti}u_{j}+g^{ij}B_{i}u_{tj}+B_{t}\triangle
u+B\frac{\partial}{\partial t}(\triangle u)\right)ud\mu\\
&\qquad+\int_{M_t}g^{ij}(Bu_i)_j\left(\frac{\partial
u}{\partial{t}}d\mu+u\frac{\partial}{\partial t}(d\mu)\right).
\end{split}
\end{equation}

Meanwhile, calculating directly
\begin{equation}\label{2.10}
\begin{split}
\frac{\partial}{\partial t}(\triangle u)=\frac{\partial}{\partial
t}(g^{ij}\nabla_{i}\nabla_{j}u)&=g^{ij}\frac{\partial}{\partial
t}(\nabla_{i}\nabla_{j}u)+\frac{\partial
g^{ij}}{\partial t}\nabla_{i}\nabla_{j}u
\\&=g^{ij}\left(\frac{\partial^{2}u_t}{\partial x_{i}\partial
x_{j}}-\frac{\partial\Gamma^{k}_{ij}}{\partial t}\frac{\partial
u}{\partial x_k}-\Gamma^{k}_{ij}\frac{\partial u_t}{\partial
x_k}\right)+\frac{\partial g^{ij}}{\partial
t}\nabla_{i}\nabla_{j}u\\&=g^{ij}\left(\nabla_{i}\nabla_{j}u_t-\frac{\partial\Gamma^{k}_{ij}}{\partial
t}\frac{\partial u}{\partial
x_k}\right)+\frac{\partial g^{ij}}{\partial t}\nabla_{i}\nabla_{j}u.
\end{split}
\end{equation}

From Lemma 2.2 (i), we immediately have
$$\frac{\partial g^{ij}}{\partial t}=-2fh^{ij},$$
which yields
\begin{equation*}
\begin{split}
\frac{\partial\Gamma^{k}_{ij}}{\partial
t}&=\frac{1}{2}g^{kl}\left(\frac{\partial}{x_i}(\frac{\partial
g_{lj}}{\partial t})+\frac{\partial}{x_j}(\frac{\partial
g_{il}}{\partial t})-\frac{\partial}{x_l}(\frac{\partial
g_{ij}}{\partial
t})\right)\\&=\frac{1}{2}g^{kl}\left(\frac{\partial}{x_i}(2fh_{jl})+\frac{\partial}{x_j}(2fh_{il})-\frac{\partial}{x_l}(2fh_{ij})\right).
\end{split}
\end{equation*}
Observing that the second fundamental form $h_{ij}$ is Codazzi in space form, then
\begin{equation*}
\begin{split}
g^{ij}\frac{\partial\Gamma^{k}_{ij}}{\partial
t}&=g^{ij}g^{kl}\left(2f_{i}h_{jl}+fh_{ij,l}-f_{l}h_{ij}\right)\\&=2f_{i}h_{lj}g^{ij}g^{kl}+fH_{l}g^{kl}-f_{l}g^{kl}H,
\end{split}
\end{equation*}
therefore, \eqref{2.10} becomes
\begin{equation}\label{2.11}
\begin{split}
\frac{\partial}{\partial t}(\triangle
u)=-2fh_{kl}u_{ij}g^{ij}g^{kl}+\triangle u_t-2u_{k}f_{i}h_{lj}g^{ij}g^{kl}+u_{k}f_{l}g^{kl}H-fH_{l}u_{k}g^{kl}.
\end{split}
\end{equation}

Putting \eqref{2.11} into \eqref{2.9} and rearranging we obtain
\begin{equation*}
\begin{split}
-\frac{\partial \lambda_{1,p}(u,t)}{\partial
t}&=-2\int_{M_t}fuh^{ij}(Bu_{j})_{i}d\mu+\int_{M_t}ug^{ij}(B_{t}u_{j})_{i}d\mu+\int_{M_t}ug^{ij}(Bu_{tj})_{i}d\mu\\
&\qquad-\int_{M_t}Bu\left(2f_{i}u_{j}h^{ij}+fH^{i}u_{i}-Hf^{i}u_{i}\right)d\mu+\int_{M_t}ug^{ij}(Bu_{i})_{j}\left(u_{t}d\mu+\frac{\partial}{\partial
t}(d\mu)\right).
\end{split}
\end{equation*}

Integrating by parts for the first three terms we arrive at
\begin{equation}\label{2.12}
\begin{split}
-\frac{\partial \lambda_{1,p}(u,t)}{\partial
t}&=2\int_{M_t}\left(Bfh^{ij}u_{i}u_{j}+Bfuu_{j}h^{ij}_{\
,i}\right)d\mu-\int_{M_t}B_{t}|\nabla
u|^{2}d\mu-\int_{M_t}Bg^{ij}u_{i}u_{tj}d\mu\\
&\quad-\int_{M_t}\left(BufH^{k}u_{k}-BuHf^{k}u_{k}\right)d\mu+\int_{M_t}g^{ij}(Bu_{i})_{j}\left(u_{t}d\mu+u\frac{\partial}{\partial
t}(d\mu)\right).
\end{split}
\end{equation}

Observing that
\begin{equation*}
\begin{split}
B_{t}=\frac{\partial B}{\partial t}&=\frac{\partial}{\partial
t}(g^{ij}u_{i}u_{j})^{\frac{p-2}{2}}\\&=\frac{p-2}{2}|\nabla
u|^{p-4}\left(\frac{\partial g^{ij}}{\partial
t}u_{i}u_{j}+2g^{ij}u_{it}u_{j}\right)\\&=(p-2)|\nabla
u|^{p-4}\left(-fh^{ij}u_{i}u_{j}+g^{ij}u_{it}u_{j}\right),
\end{split}
\end{equation*}
therefore
\begin{eqnarray}\label{2.13}
-\int_{M_t}B_{t}|\nabla
u|^{2}d\mu=(p-2)\int_{M_t}\left(Bfh^{ij}u_{i}u_{j}-Bg^{ij}u_{it}u_{j}\right)d\mu.
\end{eqnarray}

By substituting \eqref{2.13} into \eqref{2.12}, it follows that
\begin{equation}\label{2.14}
\begin{split}
-\frac{\partial \lambda_{1,p}(u,t)}{\partial
t}&=p\int_{M_t}Bfh^{ij}u_{i}u_{j}d\mu+2\int_{M_t}Bfuh^{ij}_{\ ,i}u_{j}d\mu-(p-1)\int_{M_t}Bg^{ij}u_{i}u_{tj}d\mu\\
&\quad-\int_{M_t}\left(BufH^{k}u_{k}-BuHf^{k}u_{k}\right)d\mu+\int_{M_t}g^{ij}(Bu_{i})_{j}\left(u_{t}d\mu+u\frac{\partial}{\partial
t}(d\mu)\right)\\&=p\int_{M_t}Bfh^{ij}u_{i}u_{j}d\mu+2\int_{M_t}Bfuh^{ij}_{\ ,i}u_{j}d\mu-\int_{M_t}\left(BufH^{k}u_{k}-BuHf^{k}u_{k}\right)d\mu\\
&\quad+\int_{M_t}g^{ij}(Bu_{i})_{j}\left(pu_{t}d\mu+u\frac{\partial}{\partial
t}(d\mu)\right),
\end{split}
\end{equation}
where we have used
$$\int_{M_t}Bg^{ij}u_{i}u_{tj}d\mu=-\int_{M_t}g^{ij}(Bu_{i})_{j}u_{t}d\mu.$$

Now we treat the last term in \eqref{2.14}. Since $\int_{M_t}|u|^{p}d\mu=1$ by assumption, we have by taking derivatives for both sides
\begin{equation*}
\begin{split}
\frac{\partial }{\partial
t}\int_{M_t}|u|^{p}d\mu=0&=\int_{M_t}\frac{\partial}{\partial t}(|u|^2)^{\frac{p-1}{2}}u_{t}d\mu+\int_{M_t}|u|^{p}(d\mu)_{t}\\
&=\int_{M_t}|u|^{p-2}u\left(pu_{t}d\mu+u\frac{\partial}{\partial
t}(d\mu)\right).
\end{split}
\end{equation*}

On the other hand, if $u=u(x,t)$ is the eigenfunction of $\lambda_{1,p}(t)$ on $M_{t}$ at time $t$, we have
$$
|u|^{p-2}u=-\frac{1}{\lambda_{1,p}(t)}\triangle_{p}u=-\frac{1}{\lambda_{1,p}(t)}g^{ij}\nabla_{i}(B\nabla_{j}u).
$$
As a result
\begin{equation*}
\begin{split}
\int_{M_t}|u|^{p-2}u\left(pu_{t}d\mu+u\frac{\partial}{\partial
t}(d\mu)\right)&=-\frac{1}{\lambda_{1,p}(t)}\int_{M_t}g^{ij}(Bu_{i})_{j}\left(pu_{t}d\mu+u\frac{\partial}{\partial
t}(d\mu)\right)\\&=0,
\end{split}
\end{equation*}
which says that, the last term of \eqref{2.14} disappears.

Thus \eqref{2.14}  becomes
\begin{equation*}
\begin{split}
\frac{\partial \lambda_{1,p}(u,t)}{\partial t}
&=-p\int_{M_t}Bfh^{ij}u_{i}u_{j}d\mu-2\int_{M_t}Bfuh^{ij}_{\ ,i}u_{j}d\mu+\int_{M_t}\left(BufH^{k}u_{k}-BuHf^{k}u_{k}\right)d\mu\\
&=-p\int_{M_t}Bfh^{ij}u_{i}u_{j}d\mu-\int_{M_t}BfuH^{k}u_{k}d\mu-\int_{M_t}BHuf^{k}u_{k}d\mu.
\end{split}
\end{equation*}

Integrating by parts for the last term, we have
\begin{equation*}
\begin{split}
\int_{M_t}BHuf^{k}u_{k}d\mu&=-\int_{M_t}B^{k}Hufu_{k}d\mu-\int_{M_t}BHu^{k}fu_{k}d\mu
-\int_{M_t}BH^{k}ufu_{k}d\mu-\int_{M_t}BHuf\triangle ud\mu\\
&=-\int_{M_t}Hufg^{ij}(Bu_{j})_{i}d\mu-\int_{M_t}BHf|\nabla
u|^{2}d\mu-\int_{M_t}BufH^{k}u_{k}d\mu.
\end{split}
\end{equation*}

Combining the above two equations, we finally have
\begin{equation*}
\begin{split}
\frac{\partial \lambda_{1,p}(u,t)}{\partial t}
&=-p\int_{M_t}Bfh^{ij}u_{i}u_{j}d\mu+\int_{M_t}fHu\triangle_{p}ud\mu+\int_{M_t}fHB|\nabla u|^{2}d\mu\\
&=-p\int_{M_t}Bfh^{ij}u_{i}u_{j}d\mu-\lambda_{1,p}(t)\int_{M_t}fH|u|^{p}d\mu+\int_{M_t}fH|\nabla
u|^{p}d\mu.
\end{split}
\end{equation*}
This is \eqref{evp} of proposition 2.1, and when $p=2$, \eqref{ev2} can be obtained similarly. $\hfill \square$

\section{Properties of inverse mean curvature flow}

Let $M^n$ be a smooth closed $n$-dimensional hypersurface in an Euclidean space $\mathbb{R}^{n+1}$, given by a smooth embedding $X_0:\mathbb{S}^n\rightarrow \mathbb{R}^{n+1}$. The inverse mean curvature flow is the initial value problem
\begin{eqnarray}\label{flow1}
\frac{\partial X}{\partial t}=\frac{1}{H}\nu,  ~~X(\cdot,0)=X_0(\cdot).
\end{eqnarray}
If $(M^{n}_t, g_{ij}(t))$ is the solution of flow $\eqref{flow1}$, the evolution equations of geometric quantities on $M^n_t$ can be easily obtained from Lemma 2.2, or refer \cite{Huisken08} directly.

\begin{lemma}\label{EQ1}
Under the inverse mean curvature flow $\eqref{flow1}$ , we have
\vskip.1in

$(i)$ $\frac{\partial g_{ij}}{\partial
t}=\frac{2}{H}h_{ij},~~\frac{\partial g^{ij}}{\partial
t}=-\frac{2}{H}h^{ij}$; \vskip.1in

 $(ii)$  $\frac{\partial h_{ij}}{\partial t}=\frac{1}{H^2}\nabla_i\nabla_j H-\frac{2}{H^3}\nabla_i H\nabla_j H+\frac{1}{H}(h^2)_{ij}$;\vskip.1in

 $(iii)$ $\frac{\partial H}{\partial t}=\frac{\triangle
        H}{H^2}-\frac{2}{H^3}|\nabla H|^2-\frac{|A|^2}{H}$; \vskip.1in

 $(iv)$ $\frac{\partial}{\partial t}(\frac{1}{H})=\frac{1}{H^2}\triangle(\frac{1}{H^2})+\frac{|A|^2}{H^2}\frac{1}{H}$; \vskip.1in

 $(v)$ $\frac{\partial\nu}{\partial t}=-\frac{\nabla H}{H^2}$.
\end{lemma}

For the rescaled hypersurface $\tilde{X}(t)=e^{-\frac{t}{n}}X(t)$, we denote the corresponding geometric quantities with a tilde. The following  relations can be obtained as in \cite{Huisken84}

\begin{alignat*}{3}
    \tilde{g}_{ij}&=g_{ij}e^{-\frac{2t}{n}}, &\qquad \tilde{g^{ij}}&=e^{\frac{2t}{n}}g^{ij}, \qquad &\tilde{h_{ij}}&=e^{-\frac{t}{n}}h_{ij},\\
    \tilde{H}&=e^{\frac{t}{n}}H,  &\qquad |\tilde A|^2&=e^{\frac{2t}{n}}|A|^2,  \qquad &d\tilde{\mu}&=e^{-t}d\mu.
\end{alignat*}

The rescaled hypersurface $\tilde{M}_t$ satisfies
\begin{equation}\label{flow2}
\frac{\partial \tilde{X}}{\partial
t}=-\frac{1}{n}\tilde{X}+\frac{1}{\tilde{H}}\nu.
\end{equation}

Thus, the following evolution equations for the rescaled flow \eqref{flow2} can be given by direct calculations.

\begin{lemma}\label{EQ2}
For the rescaled flow $\eqref{flow2}$, we have \vskip.1in
$(i)$ $\frac{\partial\tilde{g}_{ij}}{\partial
t}=2\left(\frac{\tilde h_{ij}}{\tilde{H}}-\frac{1}{n}\tilde g_{ij}\right)~~
\frac{\partial\tilde g^{ij}}{\partial
t}=2\left(\frac{1}{n}\tilde g^{ij}-\frac{\tilde h^{ij}}{\tilde
H}\right)$; \vskip.1in

 $(ii)$  $\frac{\partial}{\partial t}\tilde H=\frac{1}{\tilde H^2}\tilde \triangle\tilde H-\frac{1}{\tilde H^3}|\nabla\tilde H|^2+\frac{1}{\tilde H}\left(\frac{1}{n}\tilde H^2-|\tilde A|^2\right)$; \vskip.1in

$(iii)$ $\frac{\partial}{\partial t}\left(\frac{1}{\tilde H}\right)=\frac{1}{\tilde {H}^2}\tilde{\triangle}\left(\frac{1}{\tilde H}\right)
+\left(\frac{1}{n}+\frac{|\tilde
A|^2}{\tilde{H}^2}\right)\left(\frac{1}{\tilde H}\right)$; \vskip.1in

$(iv)$ $\frac{\partial\tilde{\Gamma}^k_{ij}}{\partial
t}\tilde g^{ij}=-\frac{2}{n}\tilde g^{ij}\tilde{\Gamma}^k_{ij}+
\frac{1}{\tilde{H}}\tilde g^{kl}\tilde H_l
-\frac{1}{\tilde H^2}\tilde g^{kl}\left(2\tilde h_{il}\tilde H_j\tilde g^{ij}-\tilde H\tilde H_l\right)$; \vskip.1in

 $(v)$ $\frac{\partial}{\partial t}d\tilde{\mu}=0$;\\
 where $d\tilde{\mu}$ is the volume element on $\tilde{M}^n_t$.
\end{lemma}

The long-time existence and convergence about flow \eqref{flow1} has been given by Gerhardt in \cite{Gerhardt Rn}, which is the foundation of our conclusion in this paper.

\begin{theorem}\cite{Gerhardt Rn}\label{th1.2}
Let $M^n$ be a compact, star-shaped $C^{2,\alpha}$ hypersurface in $\mathbb{R}^{n+1}$, given by an embedding
$$X_0:\mathbb{S}^n\rightarrow \mathbb{R}^{n+1},$$
then the inverse mean curvature flow \eqref{flow1} defined on $\mathbb{S}^n\times \mathbb{R}_+$ with $X_0(\mathbb{S}^n)=M^n$ has a unique solution of class $C^{2,\alpha}$, where $\nu$ is the outward unit normal of hypersurface $M_t=X(\mathbb{S}^n,t)$. And the rescaled hypersurface
\begin{eqnarray*}
\tilde{X}=e^{-\frac{t}{n}}X
\end{eqnarray*}
converge exponentially fast to a uniquely determined sphere with radius
\begin{eqnarray*}
R=\left(\frac{|M|}{|\mathbb{S}^n|}\right)^{\frac{1}{n}}.
\end{eqnarray*}
\end{theorem}

\begin{remark}
In fact, Gerhardt proved a more general result for the flow $\frac{\partial X}{\partial t}=\frac{1}{f(\sigma(\kappa_i))}\nu$, where $f(\sigma(\kappa_i))$ is a symmetric positive function homogeneous of degree one evaluated on the principle curvature $\kappa_i$ of $M_t$.
\end{remark}

\begin{lemma}\cite{Urbas}
Let $(M^{n}_t, g_{ij}(t))$ be the solution of inverse mean curvature flow \eqref{flow1}, then $$C_1e^{-t}\leq H(t)\leq C_2e^{-t},$$
where $C_1$ and $C_2$ are positive constants depending only on $n$ , $\|X_0\|$ and its
derivatives up to second order.
\end{lemma}

By the Hamilton maximum principle for tensor, we can prove the following
\begin{proposition}
Let $(M^{n}_t, g_{ij}(t))$ be the solution of inverse mean curvature flow $\eqref{flow1}$, if $h_{ij}\geq\varepsilon Hg_{ij}$ at t=0 for some $0\leq\varepsilon\leq\frac{1}{n}$, then
$$h_{ij}\geq\varepsilon Hg_{ij}$$
preserves under this flow for all time with the same $\varepsilon$.
\end{proposition}

\begin{proof}
Let $M_{ij}=h_{ij}-\varepsilon Hg_{ij}$, calculating directly by Lemma 3.1 yields
\begin{eqnarray}\label{3.1}
\begin{split}
\frac{\partial}{\partial t}M_{ij}=&\frac{1}{H^2}\nabla_i\nabla_jH-\frac{2}{H^3}\nabla_iH\nabla_jH+\frac{1}{H}(h^2)_{ij}
-\frac{\varepsilon}{H^2}\triangle Hg_{ij}\\&+\frac{2\varepsilon}{H^3}|\nabla H|^2g_{ij}+\frac{|A|^2}{H}\varepsilon g_{ij}
-2\varepsilon h_{ij}.
\end{split}
\end{eqnarray}

On the other hand, combining the  Simons' type identity
$$
\triangle h_{ij}=\nabla _i\nabla_jH-h_{ij}|A|^2+H(h^2)_{ij},
$$
we have
\begin{equation}\label{3.2}
\begin{split}
\triangle M_{ij}&=\triangle h_{ij}-\varepsilon \triangle Hg_{ij}\\
&=\nabla_i\nabla_jH-h_{ij}|A|^2+H(h^2)_{ij}-\varepsilon\triangle Hg_{ij}.
\end{split}
\end{equation}

Combination of \eqref{3.1} and \eqref{3.2} gives
\begin{eqnarray*}
\begin{split}
\frac{\partial}{\partial t}M_{ij}-\frac{1}{H^2}\triangle M_{ij}=&-\frac{2}{H^3}\nabla_iH\nabla_jH+\frac{2\varepsilon}{H^3}|\nabla H|^2g_{ij}
+\frac{|A|^2}{H}M_{ij}\\&+\frac{2\varepsilon}{H}\left(|A|^2g_{ij}-Hh_{ij}\right).
\end{split}
\end{eqnarray*}

Notice that
$$
\nabla H=\frac{g^{ij}\nabla M_{ij}}{1-n\varepsilon},
$$
and set
$$
N_{ij}=\frac{2\varepsilon}{H}\left(|A|^2g_{ij}-Hh_{ij}\right).
$$

If
$Y=\{Y^i\}$ is the null-eigenvector of $M_{ij}$, i.e., $M_{ij}Y^j=0$, equivalently,
$$h_{ij}Y^j=\varepsilon HY_i,$$
then since $0\leq\varepsilon\leq\frac{1}{n}$ by assumption, we have by Lemma 3.5
$$
N_{ij}Y^iY^j=\frac{2\varepsilon}{H}\left(|A|^2-\varepsilon H^2\right)|Y|^2\geq 0.
$$
Hence, the proof is completed by the Hamilton maximum principle.
\end{proof}

The result in Proposition 3.6 is also suitable for the rescaled inverse mean curvature flow \eqref{flow2} only multiplying a positive factor on both sides, i.e.,
\begin{equation}\label{ineq2}
\tilde{h}_{ij}-\varepsilon\tilde{H}\tilde{g}_{ij}\geq 0
\end{equation}
is also preserved under the rescaled inverse mean curvature $\eqref{flow2}$.

In order to study the first nonzero eigenvalues under the inverse mean curvature flow \eqref{flow1} and the rescaled flow \eqref{flow2} more precisely, we take
$$
\varepsilon(t)=\frac{1}{n}-\exp(-\alpha
t+\beta),~~~~0\leq\frac{\alpha}{2}\leq\frac{1}{n},~~\beta=\ln(\frac{1}{n}-\frac{\alpha}{2}).
$$

Notice that $\varepsilon(t)$ is also well defined for $\frac{\alpha}{2}=\frac{1}{n}$. We {\bf claim} that $\varepsilon(t)$ has the following properties

 $(i)$ $0\leq\varepsilon(t)\leq\frac{1}{n}$;

 $(ii)$ $\varepsilon(t)\nearrow\frac{1}{n}$, as $t\rightarrow \infty$;

 $(iii)$ $0\leq\varepsilon(t)+\frac{\varepsilon'(t)}{2\varepsilon(t)}\leq\frac{1}{n}$
 for all time $t\in [0,~\infty)$.

The properties $(i)$ and $(ii)$ are obvious from the definition of $\varepsilon(t)$, we will check  $(iii)$ below. In fact, as $0<\exp(-\alpha t)\leq 1$ in the interval $[0, \infty)$,
$$\exp(-\alpha t)(\frac{1}{n}-\frac{\alpha}{2})\leq\frac{1}{n}-\frac{\alpha}{2},$$
that is
$$\exp(-\alpha t+\beta)\leq\frac{1}{n}-\frac{\alpha}{2}.$$
Multiplying both sides  with $\exp(-\alpha t+\beta)$ gives
\begin{equation}\label{epslt}
\frac{1}{n}\exp(-\alpha t+\beta)-\frac{\alpha}{2}\exp(-\alpha t+\beta)\geq\exp2(-\alpha t+\beta).
\end{equation}

Notice that
$$\varepsilon'(t)=\alpha\exp(-\alpha t+\beta),$$
and
$$\varepsilon^2(t)=\frac{1}{n^2}+\exp2(-\alpha t+\beta)-\frac{2}{n}\exp(-\alpha t+\beta).$$
Rearranging \eqref{epslt} gives
\begin{eqnarray*}
\begin{split}
\frac{2}{n}\left(\frac{1}{n}-\exp(-\alpha t+\beta)\right)\geq & 2\left(\frac{1}{n^2}-\frac{2}{n}\exp(-\alpha t+\beta)+\exp2(-\alpha t+\beta)\right)\\ &+\frac{d}{dt}\left(\frac{1}{n}-\exp(-\alpha t+\beta)\right),
\end{split}
\end{eqnarray*}
namely
$$2\varepsilon^2(t)+\varepsilon'(t)\leq\frac{2}{n}\varepsilon(t),$$
which proves the property $(iii)$, and the claim follows.

Let
$$
\overline{M}_{ij}=h_{ij}-\varepsilon(t)Hg_{ij}.
$$
Calculating the evolution of $\overline{M}_{ij}$ under inverse mean curvature flow \eqref{flow1} similarly as in the proof of Proposition 3.6 we have
\begin{equation*}
\begin{split}
\frac{\partial}{\partial t}\overline{M}_{ij}-\frac{1}{H^2}\triangle \overline{M}_{ij}
=&-\frac{2}{H^3}\nabla_iH\nabla_jH+\frac{2\varepsilon(t)}{H^3}|\nabla
H|^2g_{ij}
+\frac{|A|^2}{H}\overline{M}_{ij}\\
&+\frac{2\varepsilon(t)}{H}\left(|A|^2g_{ij}-Hh_{ij}\right)-\varepsilon'(t)Hg_{ij}.
\end{split}
\end{equation*}

Set
$$
\overline{N}_{ij}=\frac{2\varepsilon(t)}{H}\left(|A|^2g_{ij}-Hh_{ij}\right)-\varepsilon'(t)Hg_{ij},
$$
and let $\overline{Y}$ be the null-eigenvector of $\overline{M}_{ij}$, then
$$
\overline{N}_{ij}\overline{Y}^i\overline{Y}^j=\frac{2\varepsilon(t)}{H}\left[|A|^2-\left(\varepsilon(t)+\frac{\varepsilon'(t)}{2\varepsilon}\right)H^2\right]|\overline{Y}|^2.
$$

By the property  $(iii)$ of $\varepsilon(t)$ and Lemma 3.5, we also have the following result by Hamilton maximum principle.

\begin{theorem}
Let $(M^{n}_t, g_{ij}(t))$ be the solution of inverse mean flow $\eqref{flow1}$, assuming $H>0$ and $h_{ij}\geq\varepsilon(0)Hg_{ij}$ at $t=0$ with $0\leq\varepsilon(0)=\frac{\alpha}{2}\leq\frac{1}{n}$, then
$$
h_{ij}\geq\varepsilon(t)Hg_{ij}
$$
remains true for all time $t\in[0,~\infty)$.
\end{theorem}

Similarly,
\begin{equation}\label{3.12}
\tilde{h}_{ij}\geq \varepsilon(t)\tilde{H}\tilde{g}_{ij}
\end{equation}
is also preserved under the rescaled inverse mean curvature
$\eqref{flow2}$.

\section{Monotonicity of the eigenvalues under inverse mean curvature flow}

Before giving the proof of Theorem 1.1, we firstly discuss the differentiability of $\lambda_{1,p}(t)$ and the monotonicity of  $\lambda_{1,p}(t)$ and $\lambda_{1}(t)$ under the inverse mean curvature flow \eqref{flow1} and the rescaled inverse mean curvature flow \eqref{flow2}, and obtain the following consequence.
\begin{proposition}
Let $(M^n_t,g(t))$ be a solution of the inverse mean curvature flow \eqref{flow1}, and  $\lambda_{1,p}(t)$ and  $\lambda_{1}(t)$ be the first nonzero closed eigenvalues of $p$-Laplace and Laplace operators on hypersurface $M_t$. Assume $H> 0$ and $h_{ij}\geq\varepsilon H g_{ij}$ at $t=0$, where $0\leq\varepsilon\leq\frac{1}{n}$. Then under the inverse mean curvature flow $\eqref{flow1}$,  $\lambda_{1,p}(t)$ is differentiable almost everywhere,  $\lambda_{1,p}(t)$ and $\lambda_{1}(t)$ are non-increasing for all time $t\in[0,~\infty)$, and they tend to zero as t tends to infinity. Moreover, under the rescaled inverse mean curvature flow \eqref{flow2}, $\tilde{\lambda}_{1,p}(t)$ is differentiable almost everywhere,  and $e^{-p(\frac{1}{n}-\varepsilon)t}\tilde{\lambda}_{1,p}(t)$ and $e^{-2(\frac{1}{n}-\varepsilon)t}\tilde{\lambda}_{1}(t)$ are non-increasing for all time.
\end{proposition}

Let $f=\frac{1}{H}$ in Proposition 2.1, the evolution equations for both eigenvalues under inverse mean curvature flow $\eqref{flow1}$ can be obtained immediately

\begin{lemma}
Let $(M^n_t,g(t))$ be the solution of inverse mean curvature flow $\eqref{flow1}$, $\lambda_1(t)$ be the first nonzero eigenvalue of Laplace operator, $\lambda_{1,p}(u,t)$ be defined in $\eqref{def2.7}$, and $u=u(x, t)$ be the eigenfunction of $\lambda_{1,p}(t)$ at time $t$, then we have
\begin{equation}\label{evo-plaplace}
\frac{d \lambda_{1,p}(u,t)}{dt}=-p\int_{ M_t}\frac{1}{H}|\nabla
u|^{p-2}u_iu_jh^{ij}d\mu.
\end{equation}

When $p=2$ and $u=u(x,t)$ is the eigenfunction of $\lambda_1(t)$ on $M^n_t$, we have
\begin{equation}\label{4.13}
\frac{d \lambda_{1}(t)}{dt}=-2\int_{
M_t}\frac{1}{H}u_iu_jh^{ij}d\mu.
\end{equation}
\end{lemma}

Combining $\eqref{evo-plaplace}$ and Proposition 3.6 gives
\begin{equation}\label{4.0}
\frac{d \lambda_{1,p}(u,t)}{dt}\leq-p\varepsilon\lambda_{1,p}(t)\leq 0,
\end{equation}
that's to say, $\lambda_{1,p}(u,t)$ is non-increasing under inverse mean curvature flow \eqref{flow1}. We will discuss the monotonicity and differentiability of the first nonzero eigenvalue $\lambda_{1,p}(t)$ of $p$-Laplace operator under inverse mean curvature flow \eqref{flow1} similarly as in \cite{Mao,zhao12}.

From \eqref{4.0} we know, at time $t_0\in[0,~\infty)$
\begin{equation}\label{4.1}
\frac{d \lambda_{1,p}(u,t)}{dt}\mid_{t_0}\leq 0.
\end{equation}
Furthermore, we also know by the definition of $\lambda_{1,p}(u,t)$ in \eqref{def2.7} that $\lambda_{1,p}(u,t)$ is smooth w.r.t. time $t$, and then the above inequality still holds for any small neighbourhood of $t_0$, i.e., for sufficiently small $\delta>0$, we have \eqref{4.1} in the interval $[t_0,t_0+\delta]$. Integrating both sides of \eqref{4.1} in the interval $[t_0,t_0+\delta]$ yields
$$
\lambda_{1,p}\left(u(\cdot,t_0+\delta),t_0+\delta\right)\leq\lambda_{1,p}(u(\cdot,t_0),t_0).
$$

On the other hand, noticing that $u=u(x,t)$ is the eigenfunction at time $t_0$, we have
$$\lambda_{1,p}(u(\cdot,t_0),t_0)=\lambda_{1,p}(t_0).$$
By the definitions of $\lambda_{1,p}(t)$ in $\eqref{df1}$ and $\lambda_{1,p}(u,t)$ in \eqref{def2.7}, we know
$$\lambda_{1,p}(u(\cdot,t_0+\delta),t_0+\delta)\geq\lambda_{1,p}(t_0+\delta).$$

Combining the above facts we obtain
$$\lambda_{1,p}(t_0)\geq\lambda_{1,p}(t_0+\delta).$$

 Since $t_0\in[0,\infty)$ is arbitrary, $\lambda_{1,p}(t)$ is non-increasing under the inverse mean curvature flow \eqref{flow1} for all times, and the differentiability can be obtained by Lebesgue's theorem.

Thus we can replace $\lambda_{1,p}(u,t)$ by $\lambda_{1,p}(t)$ in \eqref{4.0} immediately, and integrate both sides in the interval $[t_0,t_0+\delta]$ to get
\begin{equation}
\lambda_{1,p}(t)\leq\lambda_{1,p}(0)e^{-p\varepsilon t},
\end{equation}
that is, a decreasing upper bound for $\lambda_{1,p}(t)$ is obtained along inverse mean curvature flow \eqref{flow1}, and it also means that $\lambda_{1,p}(t)$ tends to zero when $t$ tends to infinity.

The corresponding conclusion to $\lambda_{1}(t)$ can be obtained from \eqref{4.13} and Proposition 3.6 similarly, so we complete the proof for the first part of Proposition 4.1.

Obviously, we can't obtain the evolution equations for both eigenvalues under the rescaled flow \eqref{flow2} from Proposition 2.1, but it is not difficult from Lemma 3.2 by a similar process of the proof of Proposition 2.1.

\begin{lemma}
Let $(\tilde{M}^n_t,\tilde{g})$ be the solution of the rescaled flow $\eqref{flow2}$, $\tilde{\lambda}_{1,p}(t)$ and $\tilde{\lambda}_1(t)$ be the eigenvalues of $p$-Laplace and Laplace operators on hypersurface $\tilde{M}^n_t$, $\tilde{u}=\tilde{u}(x,t)$ be the eigenfunction of $\tilde\lambda_{1,p}(t)$ at time $t$, $\tilde\lambda_{1,p}(\tilde{u},t)$ be defined in $\eqref{def2.7}$, then at time $t$, we have
\begin{equation}\label{4.19}
\frac{d
\tilde{\lambda}_{1,p}(\tilde{u},t)}{dt}=\frac{p}{n}\tilde{\lambda}_{1,p}(t)-p\int_{\tilde
M_t}\frac{1}{\tilde{H}}|\tilde{\nabla} \tilde
u|^{p-2}\tilde{u}_i\tilde{u}_j\tilde{h}^{ij}d\tilde\mu.
\end{equation}

When $p=2$ and $\tilde{u}=\tilde{u}(x,t)$ is the eigenfunction of $\tilde{\lambda}_1(t)$ on $\tilde{M}^n_t$, we have
\begin{equation}\label{4.20}
\frac{d
\tilde{\lambda}_{1}(t)}{dt}=\frac{2}{n}\tilde{\lambda}_{1}(t)-2\int_{\tilde
M_t}\frac{1}{\tilde{H}}\tilde{u}_i\tilde{u}_j\tilde{h}^{ij}d\tilde\mu.
\end{equation}
\end{lemma}

Assuming $\tilde{H}>0$ initially, combining $\eqref{ineq2}$ and $\eqref{4.19}$  we obtain
\begin{equation}\label{4.26}
\frac{d\tilde{\lambda}_{1,p}(\tilde{u},t)}{dt}\leq p\left(\frac{1}{n}-\varepsilon\right)\tilde{\lambda}_{1,p}(t).
\end{equation}
By the definition of $\lambda_{1,p}(t)$ in $\eqref{df1}$ and $\lambda_{1,p}(u,t)$ in \eqref{def2.7} we know
$$\tilde{\lambda}_{1,p}(t)\leq\tilde{\lambda}_{1,p}(\tilde{u}(\cdot,t),t).$$
Again since $\varepsilon\leq\frac{1}{n}$, \eqref{4.26} becomes
\begin{equation}\label{4.27}
\frac{d\tilde{\lambda}_{1,p}(\tilde{u},t)}{dt}\leq p\left(\frac{1}{n}-\varepsilon\right)\tilde{\lambda}_{1,p}(\tilde{u},t),
\end{equation}
which is equivalent that $e^{-p(\frac{1}{n}-\varepsilon)t}\tilde{\lambda}_{1,p}(\tilde{u},t)$ is non-increasing under the rescaled inverse mean curvature flow \eqref{flow2}.

Obviously, $e^{-p(\frac{1}{n}-\varepsilon)t}\tilde{\lambda}_{1,p}(\tilde{u},t)$  is smooth w.r.t. time $t$ from the definition of $\tilde{\lambda}_{1,p}(\tilde{u},t)$ in \eqref{def2.7}, thus, at the neighbourhood of time $t_0$, namely, in the small interval $[t_0,t_0+\delta]$, we have
$$
e^{-p(\frac{1}{n}-\varepsilon)t_0}\tilde{\lambda}_{1,p}\left(\tilde{u}(\cdot,t_0),t_0\right)\geq e^{-p(\frac{1}{n}-\varepsilon)(t_0+\delta)}\tilde{\lambda}_{1,p}\left(\tilde{u}(\cdot,t_0+\delta),t_0+\delta\right).
$$
Noticing that $\tilde{u}=\tilde{u}(x,t_0)$ is the eigenfunction at time $t_0$ we have
$$\tilde{\lambda}_{1,p}(\tilde{u}(\cdot,t_0),t_0)=\tilde{\lambda}_{1,p}(t_0),$$
and
$$\tilde{\lambda}_{1,p}(\tilde{u}(\cdot,t_0+\delta),t_0+\delta)\geq\tilde{\lambda}_{1,p}(t_0+\delta)$$
by the definitions of $\lambda_{1,p}(t)$ in $\eqref{df1}$ and $\lambda_{1,p}(u,t)$ in \eqref{def2.7}.
Combination of the above inequalities yields
$$
e^{-p(\frac{1}{n}-\varepsilon)t_0}\tilde{\lambda}_{1,p}\left(t_0\right)\geq e^{-p(\frac{1}{n}-\varepsilon)(t_0+\delta)}\tilde{\lambda}_{1,p}\left(t_0+\delta\right),
$$
i.e., $e^{-p(\frac{1}{n}-\varepsilon)t}\tilde{\lambda}_{1,p}(t)$ is non-increasing under the rescaled inverse mean curvature flow \eqref{flow2}, and the differentiability of $\tilde{\lambda}_{1,p}(t)$ can be proved by Lebesgue's theorem.

The responding result to $\tilde\lambda_1(t)$ can be obtained by a more easier procedure. So we complete the proof of Proposition 4.1. $\hfill\square$

\vspace{3mm}
\noindent
\emph{\textbf{Proof of Theorem \ref{th1.1}}}

As we have proved the differentiability of $\tilde{\lambda}_{1,p}(t)$ under the rescaled inverse mean curvature flow \eqref{flow2}, we can replace $\tilde{\lambda}_{1,p}(\tilde{u},t)$ with $\tilde{\lambda}_{1,p}(t)$ in \eqref{4.26}, and combine \eqref{3.12} in Theorem 3.7 to give

\begin{equation}\label{4.29}
\frac{d\tilde{\lambda}_{1,p}(t)}{dt}\leq p\left(\frac{1}{n}-\varepsilon(t)\right)\tilde{\lambda}_{1,p}(t).
\end{equation}
Integrating both sides of \eqref{4.29} in $[0,~t]$, we have
\begin{equation*}
\begin{split}
\tilde{\lambda}_{1,p}(t)&\leq\tilde\lambda_{1,p}(0)\exp\left[p\int^t_0\left(\frac{1}{n}-\varepsilon(t)\right)dt\right]\\&
=\tilde\lambda_{1,p}(0)\exp\left[-\frac{p}{\alpha}\exp(-\alpha t+\beta)+\frac{p}{\alpha}\left(\frac{1}{n}-\frac{\alpha}{2}\right)\right].
\end{split}
\end{equation*}

We know from Theorem 3.3 that $\tilde M^n_t$ tends to some sphere $\mathbb{S}^n(R)$ which has the same area with $\tilde M^n_t$ when  time $t$ tends to infinity, hence, at $t=\infty$
\begin{equation}\label{4.30}
\tilde{\lambda}_{1,p}(\mathbb{S}^n(R))\leq\tilde{\lambda}_{1,p}(\tilde M_0)\exp\left[\frac{p}{\alpha}\left(\frac{1}{n}-\frac{\alpha}{2}\right)\right].
\end{equation}

From above all process, the equality holds iff $h_{ij}=\varepsilon(0)Hg_{ij}=\frac{\alpha}{2}Hg_{ij}$ at the initial time, i.e., $\frac{\alpha}{2}=\frac{1}{n}$, which means the initial hypersurface $\tilde M^n_0=M^n$ is the sphere $ \mathbb{S}^n(R)$.

The corresponding result to $\tilde{\lambda}_1(t)$ can be obtained by an almost same procedure.
So we complete the proof of Theorem \ref{th1.1}. $\hfill\square$

\end{document}